\documentclass[12pt,oneside]{article}

\usepackage{geometry} 
\usepackage{amssymb}

\usepackage{amsmath}

\usepackage{amsfonts}

\usepackage{longtable}

\usepackage{verbatim}

\newtheorem{theorem}{Theorem}[section]
\newtheorem{lemma}[theorem]{Lemma}

\newtheorem{proposition}[theorem]{Proposition}

\newtheorem{corollary}[theorem]{Corollary}

\newtheorem{definition}[theorem]{Definition}

\newtheorem{fact}[theorem]{Fact}

\newenvironment{proof}
{\begin{trivlist}  \item \textsc{Proof:}~} {\hfill $\Box$
\end{trivlist}}

\newenvironment{proof of claim}
{\begin{trivlist}  \item \textsc{Proof of Claim:}~} {\hfill $\Box$
\end{trivlist}}

\newenvironment{theoremno}
{\begin{trivlist} \item \textbf{Theorem:}~}{\end{trivlist}}

\newenvironment{propositionno}
{\begin{trivlist} \item \textbf{Proposition:}~}{\end{trivlist}}

\newcommand{\closure}[1]{\ensuremath{\mathrm{cl}}(#1)}
\newcommand{\interior}[1]{\ensuremath{\mathrm{int}}(#1)}

\def \diam{\operatorname{Diam}}

\def \st {\operatorname{st}}

\def \Q{\mathbb{Q}}

\def \R{R}

\def \O{\mathcal{O}}
\def \bR{\mathbb{R}}

\def\Ind#1#2{#1\setbox0=\hbox{$#1x$}\kern\wd0\hbox to 0pt{\hss$#1\mid$\hss}
\lower.9\ht0\hbox to 0pt{\hss$#1\smile$\hss}\kern\wd0}

\def\Notind#1#2{#1\setbox0=\hbox{$#1x$}\kern\wd0\hbox to 0pt{\mathchardef
\nn=12854\hss$#1\nn$\kern1.4\wd0\hss}\hbox to
0pt{\hss$#1\mid$\hss}\lower.9\ht0 \hbox to
0pt{\hss$#1\smile$\hss}\kern\wd0}

\newcommand{\ma}{\mathfrak{m}}

\newcommand{\hdim}{\dim_{\mathcal{H}}}

\begin{document}

\title{The Hausdorff dimension of metric spaces definable in o-minimal expansions of the real field\footnote{This work was partially supported by a grant from the Simons Foundation (Grant 318364 to Jana Ma\v{r}\'{i}kov\'{a})}}
\author{Jana Ma\v{r}\'ikov\'{a} and Erik Walsberg}

\AtEndDocument{\bigskip{\footnotesize%
  \textsc{Department of Mathematics, Western Illinois University, Macomb, Illinois} \par  
  \textit{E-mail address}: \texttt{j-marikova@wiu.edu} \par
  \addvspace{\medskipamount}
  \textsc{Institute of Mathematics, Hebrew University of Jerusalem, Givat Ram, Jerusalem} \par
  \textit{E-mail address}: \texttt{erikw@math.ucla.edu}
}}

\maketitle

\begin{abstract}
Let $R$ be an o-minimal expansion of the real field.
We show that the Hausdorff dimension of an $\R$-definable metric space is an $\R$-definable function of the parameters defining the metric space.
We also show that the Hausdorff dimension of an $\R$-definable metric space is an element of the field of powers of $\R$.
The proof uses a basic topological dichotomy for definable metric spaces due to the second author, and the work of the first author and Shiota on measure theory over nonarchimedean o-minimal structures.
\end{abstract}

\begin{section}{Introduction}
Let $R$ be an o-minimal expansion of $\overline{\mathbb{R}} = (\mathbb R, <, +, \times )$, i.e. of the real field.
Throughout, by ``definable'' we mean ``definable in $R$ (possibly with parameters)'', unless stated otherwise.  If $M$ is an expansion of an ordered abelian group, then we write $M^>$ instead of $M^{>0}$, and $M^{\geq }$ instead of $M^{\geq 0}$.
We let $\Lambda$ be the \textbf{field of powers} of $\R$, i.e. the set of $r \in \bR$ such that $t^r$ is a definable function of $t \in \bR$.
Given a definable set $A$ we let $\dim(A)$ be the o-minimal dimension of $A$.
This agrees with the topological dimension of $A$.  We use the term topological dimension here in the sense of Hieronymi, Miller \cite{millerpersonal} (see Introduction of \cite{millerpersonal}), where it refers to small inductive dimension.
A \textbf{definable metric space} is a definable set $X$ equipped with a definable metric $d: X^2 \rightarrow \bR^>$.
Some basic facts about definable metric spaces were established in \cite{erikthesis}.

The Hausdorff dimension of a definable set agrees with its topological dimension.
In fact, it was recently shown in \cite{millerpersonal} that this holds for closed sets in any model-theoretically tame first order expansion of the real field:
\begin{theoremno}\label{theorem:hm}
Let $\mathcal M$ be an expansion of the real field which does not define the integers and let $X \subseteq \bR^n$ be a closed $\mathcal M$-definable set.
Then the Hausdorff dimension and the topological dimension of $X$ agree and equal the largest $m$ for which there is a coordinate projection $\mathbb{R}^n \rightarrow \mathbb{R}^m$ which maps $X$ onto a set with nonempty interior.
\end{theoremno}
The above theorem implies that the Hausdorff dimension of a closed $\mathcal M$-definable set is a natural number, and is furthermore a definable function of the parameters defining the set.
In contrast there are $R$-definable metric spaces with fractional Hausdorff dimension.
Let $0 < r < 1$ be an element of $\Lambda$.
Then $d(x,y) = |x - y|^r$ is a definable metric on $[0,1]$ with Hausdorff dimension $\frac{1}{r}$.
Varying $r$ produces an $\overline{\mathbb{R}}_{\exp}$-definable family of metric spaces whose elements have infinitely many distinct Hausdorff dimensions.
There are interesting examples of semialgebraic metric spaces, coming from  nilpotent geometry, with Hausdorff dimension strictly greater then topological dimension.
The simplest example is the Heisenberg group.
The Heisenberg group $H$ is the group of matrices of the form:
\[\begin{pmatrix}
 1 & x & z \\
  0 & 1 & y \\
  0 & 0 & 1
\end{pmatrix} \quad \text{for } x,y,z \in \bR.\]
We equip $H$ with a norm $\| \|_H$ by declaring the norm of the matrix above to be $[x^4+y^4+z^2]^{\frac{1}{4}}$.
We equip $H$ with a left-invariant metric $d_H$ by declaring $d_H(A,B) = \|A^{-1}B\|_H$ for all $A,B \in H$.
The Heisenberg group has topological dimension three and the Hausdorff dimension of $(H,d_H)$ is four.
The Heisenberg group is the simplest non-abelian \textbf{Carnot Group}.
Any Carnot group can be equipped with a left-invariant semialgebraic metric in a manner similar to the Heisenberg group.
If the Carnot group is non-abelian, then the Hausdorff dimension of the associated metric space is strictly greater then the topological dimension of the Carnot group.
See 5.5 of \cite{erikthesis} for more information and references.
Subriemannian metrics give more examples of interesting $\overline{\mathbb{R}}_{an}$-definable metric spaces whose Hausdorff dimension is strictly larger then topological dimension, see 5.9 of \cite{erikthesis}.
In this paper we prove the following:
\begin{theorem}\label{theorem:main}
Let $\mathcal{X} = \{ (X_\alpha,d_\alpha) : \alpha \in \bR^l \}$ be a definable family of metric spaces.
Then the Hausdorff dimension of $(X_\alpha,d_\alpha)$ is a definable function of $\alpha$ taking values in $\Lambda\cup \{\infty\}$. 
If $R$ is polynomially bounded then the Hausdorff dimension of the elements of $\mathcal{X}$ takes only finitely many values.
\end{theorem}
Theorem~\ref{theorem:main} implies that the Hausdorff dimension of an $\overline{\mathbb{R}}_{an}$-definable metric space is a rational number.
Theorem~\ref{theorem:main} was proven for $\overline{\mathbb{R}}_{an}$ in Proposition 11.3.2 and Proposition 11.3.3 of \cite{erikthesis}.
That proof was an application of the theorem of Comte, Lion and Rolin~\cite{CLR} on volumes of $\overline{\mathbb{R}}_{an}$-definable sets.
The present proof relies on weaker but more general facts about volumes of definable sets.
Here is an outline of the proof.
In Section~\ref{section:hausdorff} we define Hausdorff dimension and gather some basic facts from metric geometry.
The crucial result of metric geometry that we use is the following version of the Mass Distribution principle:
\begin{propositionno}
Let $(X,d)$ be a separable metric space and let $\mu$ be a finite Borel measure on $X$ which gives every open set positive measure.  For $p\in X$ and $t\in \mathbb{R}^{>}$, let $B(p,t)\subseteq X$ be the closed ball of radius $t$ centered at $p$.
Suppose that the limit
$$ \phi(p) := \lim_{ t \rightarrow 0^+} \frac{ \log[ \mu B(p,t) ]}{\log(t)} \in \bR\cup\{\infty\} \quad \text{exists for all }p \in X, $$
and that $\phi : X \rightarrow \bR\cup\{\infty\}$ is continuous.
Then the Hausdorff dimension of $(X,d)$ is the supremum of the range of $\phi$.
\end{propositionno}
Let $(X,d)$ be a definable metric space.
The following basic dichotomy for definable metric spaces, 9.0.1 of \cite{erikthesis}, will allow us to apply the previous proposition to $(X,d)$.
\begin{theoremno}
One of the following holds:
\begin{enumerate}
\item There is an infinite definable subset $A \subseteq X$ such that $(A,d)$ is discrete.
\item There is a definable homeomorphism between $(X,d)$ and a definable set equipped with its euclidean topology.
\end{enumerate}
\end{theoremno}
If $(1)$ holds, then the Hausdorff dimension of $(X,d)$ is infinite.
Therefore we may assume that the $d$-topology on $X$ agrees with the euclidean topology.
Omitting technical details, we prove Theorem~\ref{theorem:main} by applying the following proposition, a consequence of mild extensions of results in \cite{MS}, to the definable family $\{ B(p,t) : (p,t) \in X \times \bR^> \}$ of balls in $(X,d)$.
This is Proposition~\ref{keyf} below.
\begin{propositionno}
Let $\lambda$ be $k$-dimensional Lebesgue measure.
Let $X$ be a definable set and $\{ A_{p,t} : (p,t) \in X \times \bR^{>} \}$ be a definable family of  $k$-dimensional sets such that 
$ \lim_{t \rightarrow 0^+} \lambda (A_{p,t}) = 0$ for all $p \in \bR^l$.  Then $$\phi (p):=\lim_{t\to 0^+}\frac{\log \lambda (A_{p,t})}{\log t} \in \bR\cup\{\infty\}$$ is a definable function of $p$ which takes values in $\Lambda\cup\{\infty\}$.
If $R$ is polynomially bounded, then $\phi$ takes only finitely many distinct values.
\end{propositionno}
For $\bar{\mathbb{R}}_{an}$, this is an easy consequence of the work of Comte, Lion and Rolin~\cite{CLR}.

\medskip\noindent
{\bf Notation and Conventions.\/}
By $k,l,m,n$ we shall always denote nonnegative integers.
We let $\mathcal{R}$ be a big elementary extension of $R$.  
The expansion of $R$ by the logarithm, $R_{log}$, is o-minimal, see \cite{speiss}).
By saturation we may view $\mathcal R$ as a substructure of an elementary expansion of $R_{log}$, and in this sense we take the logarithm of elements of $\mathcal{R}$.

By $\O$ we denote the convex hull of $\Q$ in $\mathcal R$.  Then $\O$ is a convex subring of $\mathcal R$, hence a valuation ring.  We denote its maximal ideal by $\ma$. Its residue field is $\overline{\mathbb{R}}$ with residue map $\st \colon \mathcal{O} \to \mathbb{R}$.






Let $M$ be an o-minimal field, and let 
$X,Y \subseteq M^n$ be definable.  Then we write $X \subseteq_0 Y$ iff $\dim (X\setminus Y)<n$, and $X=_0 Y$ iff $X\subseteq_0 Y$ and $Y \subseteq_0 X$. A property holds for almost all elements of $X$ if it holds for all elements of $X$ outside of a definable subset of dimension $<n$.  By a box in $M^n$ we mean a definable set of the form $[a_1 , b_1 ]\times  \dots \times [a_n ,b_n ] \subseteq M^n$, where $a_i < b_i$ for all $i$. By $p^{n}_{m}\colon M^n \to M^m$ we denote the projection onto the first $m$ coordinates.

We say that a definable $X \subseteq \mathcal{O}^n$ is $d$-\textbf{thin} if $\dim(X) \leqslant d$ and $\interior {\st ( \pi X)}=\emptyset$ for all orthogonal projections $\pi: \mathcal{R}^n \rightarrow \mathcal{R}^d$.
In the terminology of Valette~\cite{vanish} a thin set is ``$\mathcal{O}$-thin''.
If a definable $X \subseteq \mathcal{O}^n$ is not $d$-thin then we say that it is $d$-\textbf{fat}.
Note that a definable subset of $\mathcal{O}^n$ is $n$-fat if and only if $\interior{\st X} \neq \emptyset$.

If $(X,d)$ is a metric space and $p\in X$, $t\in \mathbb{R}^>$, then $B(p,t)\subseteq X$ is the closed ball of radius $t$ centered at $p$.

\end{section}

\section{Hausdorff Dimension}\label{section:hausdorff}
\textit{Throughout this section $(X,d)$ is a metric space.}
In this section we gather some facts about Hausdorff dimension.
Given $A \subseteq X$ the diameter of $A$ is 
$$\diam_d(A) := \sup\{ d(x,y) : x,y \in A \}\in \mathbb{R}_{\infty}.$$
Fix $r \in \bR^>$.
We define the $r$-dimensional \textbf{Hausdorff measure} on $X$.
Given $\delta >0$ and $A \subseteq X$ we declare
$$ \mathcal{H}^{r}_\delta(A) = \inf \sum_{B \in \mathcal{B}} [\diam_d(B)]^r \in \bR_{\infty}. $$
where the infimum is taken over all countable collections $\mathcal{R}$ of closed balls which cover $A$ and have diameter at most $\delta$.
We define
$$ \mathcal{H}^{r}(A) = \lim_{\delta \rightarrow 0^+} \mathcal{H}^{r}_\delta(A). $$
The limit exists as $\mathcal{H}^{r}_\delta(A)$ decreases with $\delta$.
It is a classical fact that $\mathcal{H}^{r}$ gives a Borel measure on $X$.
The \textbf{Hausdorff dimension} of $(X,d)$ is 
$$ \hdim(X,d) := \sup\{r \in \bR^{>}: \mathcal{H}^r(X,d) = \infty\} \in \bR_{\infty}. $$

\begin{fact}\label{fact:separable}
If $(X,d)$ is not separable then $\hdim(X,d) = \infty$
\end{fact}

\begin{proof}
Suppose that $\hdim(X,d) < \infty$.
Then for every $\delta > 0$, $(X,d)$ admits a countable covering by balls of diameter at most $\delta$.
This implies that $(X,d)$ is separable.
\end{proof}
The next fact follows directly from the definition of Hausdorff dimension and the fact that $\mathcal{H}^r$ is a Borel measure:
\begin{fact}\label{fact:union}
Let $\{ A_i : i \in \mathbb{N} \}$ be a countable collection of Borel sets which covers $X$.
Then 
$$ \hdim(X,d) = \sup\{ \hdim(A_i,d): i \in \mathbb{N} \}. $$
\end{fact}

The following proposition is a form of the Mass Distribution Principle.
It is a variant of the results in \cite{mattila}.
As we do not know of a proof of this exact result in the metric geometry literature we refer to Proposition 3.3.6 of \cite{erikthesis} for a proof.

\begin{proposition}\label{proposition:pointwisemass}
Suppose $(X,d)$ is separable.
Let $\mu$ be a finite Borel measure on $X$.
Let $A \subseteq X$ be Borel with $\mu(A) > 0$.
Let $r > 0$.
Suppose that for every $p \in A$ 
$$ \theta_r(p) := \lim_{ t \rightarrow 0^+ } \frac{ \mu [B(p,t)] }{ t^r } \in \bR_\infty  $$
exists. 
If $\theta_r(p) > 0$ at every $p \in A$ then $\hdim(A,d) \geqslant r$.
If $\theta_r(p) < \infty$ at every $p \in A$ then $\hdim(A,d) \leqslant r$.
\end{proposition}
We make use of the following:
\begin{proposition}\label{proposition:logmassdistro}
Suppose $(X,d)$ is separable.
Let $\mu$ be a finite Borel measure on $X$ which assigns to every open set positive measure.
Suppose that for every $p \in X$, the limit
$$ \phi(p) := \lim_{ t \rightarrow 0^+ } \frac{ \log \mu [B(p,t)] }{ \log t } \in \mathbb{R}_\infty $$
exists.
If $\phi$ is continuous, then
$$ \hdim(X,d) = \sup\{ \phi(p) : p \in X \} .$$
\end{proposition}

\begin{proof}
Suppose that $\phi(p) < r$ for all $p \in X$.
Then for each $p \in X$ we have
$$ \lim_{t \rightarrow 0^+ } \frac{ \mu [B(p,t)] }{ t^r } < \infty $$
Applying Proposition~\ref{proposition:pointwisemass} we have $\hdim(X,d) \leqslant r$.
Now suppose that $\sup \{ \phi(p) : p \in X \} > r$.
As $\phi$ is continuous there is an open $U$ on which $\phi$ is bounded from below by $r$.
For each $p \in U$ we have
$$ \lim_{t \rightarrow 0^+ } \frac{ \mu [B(p,t)] }{ t^r } > 0 $$
As $U$ has positive measure, Proposition~\ref{proposition:pointwisemass} implies $\hdim(U,d) \geqslant r$.
This shows that $\hdim(X,d) \geqslant r$.
\end{proof}

\begin{section}{Miller's Dichotomy}
In this section we prove Proposition~\ref{keyf} for definable families of subsets of $\mathbb{R}$.
Let $\mathcal{A} = \{ A_{p,t} : (p,t) \in \bR^l \times \bR^> \}$ be a definable family of subsets of $\bR$.
The one-dimensional Lebesgue measure of $A_{p,t}$ is the sum of the lengths of the connected components of $A_{p,t}$ and is thus a definable function of $(p,t)$.
Corollary~\ref{def} for $\mathcal{A}$ will follow from Proposition~\ref{proposition:millerconseq}:
\begin{proposition}\label{proposition:millerconseq}
Let $\{ f_p : p \in \bR^n \}$ be a definable family of functions $\bR^{>} \rightarrow \bR^{>}$.
Then
$$ F(p) := \lim_{ t \rightarrow 0^+ } \frac{ \log f_p(t) }{ \log(t) } \in \bR_{\infty} $$
is a definable function of $p$ which takes values in the field of powers $\Lambda_\infty$.
If $\R$ is polynomially bounded, then $F$ takes only finitely many distinct values.
\end{proposition}
Proposition~\ref{proposition:millerconseq} is a corollary of two results of Miller.
The first is the fundamental dichotomy~\cite{CM}:
\begin{theorem}
Exactly one of the following holds:
\begin{enumerate}
\item $\R$ is polynomially bounded,
\item The exponential function is definable.
\end{enumerate}
\end{theorem}
The second is Proposition 5.2, p. 92, \cite{CMpower}.
\begin{fact}\label{fact:cmpower}
Suppose that $R$ is polynomially bounded, and let
$f\colon \mathbb{R}^n \times \mathbb{R} \to \mathbb{R}$ be definable.  Then there are $r_1 , \dots ,r_k \in \Lambda$ such that for each $p\in \mathbb{R}^n$, either $f_p$ is ultimately identically 0, or
$$\lim_{ t \rightarrow \infty} \frac{f_p (x)}{x^{r_i}}=c(p)$$ for some $i \in \{ 1, \dots ,k  \}$ and definable $c\colon \bR^n \to \bR^{\times}$. 
\end{fact}
Note that then, if $f_p \colon \mathbb{R}^{>} \to \mathbb{R}^{>}$ for all $p\in \mathbb{R}^n$, then there are $r_1 ,\dots ,r_k \in \Lambda$ such that for each $p\in \mathbb{R}^n$ there is a $i\in \{ 1,\dots ,k  \}$ so that: $$0<\lim_{x\to 0^+} \frac{f_p (x)}{x^{r_i }}<\infty .$$
We now prove Proposition~\ref{proposition:millerconseq}:
\begin{proof}
The proposition clearly holds if the exponential function is definable.
Suppose that $\R$ is polynomially bounded.
Fix $p \in \bR^n$.
Applying Fact~\ref{fact:cmpower} we obtain
$\lambda_1,\lambda_2 \in \bR^>$ and $r \in \bR$ are such that $\lambda_1 t^r \leqslant f_p(t) \leqslant \lambda_2 t^r$ holds for all sufficiently small $t$.
Taking $\log$ and dividing through by $\log(t)$ yields:
$$ \frac{ \log(\lambda_1) }{ \log(t) } + r \geqslant \frac{ \log f_p(t) }{ \log(t) } \geqslant \frac{ \log(\lambda_2) }{ \log(t) } + r $$
for all sufficiently small $t$.
Thus:
$$  \lim_{ t \rightarrow 0^+ } \frac{ \log f_p(t) }{ \log(t) } = r. $$
Let $r_1,\ldots,r_k \in \Lambda$ be such that for any $p \in \mathbb{R}^n$ there are
$\lambda_1,\lambda_2 \in \bR^>$ such that for some $i \in \{ 1, \dots ,k  \}$, $\lambda_1 t^{r_i} \leqslant f_p(t) \leqslant \lambda_2 t^{r_i}$ holds for all sufficiently small $t$.
For all $i \in \{ 1,\dots ,k\}$ the set of $p \in \mathbb{R}^n$ such that
$$ 0 < \lim_{t \rightarrow 0^+} \frac{f_p(t)}{t^{r_i}} < \infty$$
is definable.
It follows that
$$ \lim_{ t \rightarrow 0^+ } \frac{ \log f_p(t) }{ \log(t) } $$
is a definable function of $p \in \mathbb{R}^n$.
\end{proof}

\end{section}

\begin{section}{Measures on definable sets}\label{section:ms}
In this section we prove Proposition~\ref{keyf}.
In \cite{MS}, the authors define a finitely additive measure $\nu$ (see Definition \ref{defmeasure}) on the definable subsets of $\mathcal{O}^n$ which takes values in an ordered semiring (a quotient of $\mathcal{O}^{\geq}$) and agrees with the Lebesgue measure of $\st X$ in the case when $X$ is $n$-fat.  When $X$ is $n$-thin, the measure of an open cell $X$ agrees with the supremum of the measure of all boxes inscribed in a certain isomorphic image of $X$ (an isomorphism $X \to \phi X$ here is, roughly, a $C^1$-diffeomorphism $\phi$ such that $|\det J (x) |=1$ for almost all $x \in X$).  We first need to extend the definitions in \cite{MS} to $d$-dimensional measure.  We note that while we assume for simplicity that $\mathcal{R}$ is sufficiently saturated, this assumption is not needed for the definition of $\nu$ in \cite{MS}, nor is it needed for its $d$-dimensional version.  The only adjustment one needs to make when dropping the saturation assumption is to replace $\widetilde{\mathcal{O}}$ in Definition \ref{defmeasure} by its Dedekind completion.

Definitions~\ref{eqdef} and~\ref{defmeasure} below are from \cite{MS} (stated in slightly weaker form).  Lemma \ref{defapprox} is a consequence of results from \cite{MS}.  It will yield the desired result on limits of families of open sets (Corollary \ref{def}), which will in turn imply the result in full generality.

\smallskip\noindent
We shall use the following convention.  Suppose $M$ is an o-minimal field and $X\subseteq M^n$ is an open cell with $p^{n}_{k}X = (f_k , g_k )$ for $k=1, \dots ,n$.  Let $x\in X$ and $t\in (0,1)^n$ be such that $$x_i = (1-t_i ) f_i (x_1 , \dots ,x_{i-1}) + t_i g_i (x_1 , \dots ,x_{i-1}).$$
Then $\tau_X \colon X \to \tau_X X \subseteq M^n$ is the map $\tau_X (x) =y$, where
$$y_i = t_i (g_i (x_1 , \dots , x_{i-1}) - f_i (x_1 , \dots ,x_{i-1} ) ) .$$
We define an equivalence relation $\sim$ on $\mathcal{O}^\geqslant$:

\begin{definition}\label{eqdef} Let $x,y \in \ma^{\geqslant}$.  Then $x\sim y$ iff $y^q \leqslant x \leqslant y^p$ for all $p,q \in \mathbb{Q}^{>}$ such that $p<1<q$.  If $x,y >\ma$, then $x\sim y$ iff $\st x =\st y$.  We let $\widetilde{\mathcal{O}}$ be the quotient $\mathcal{O}^{\geq}/\sim$.
\end{definition}
Note that $\mathbb{R}\subseteq \widetilde{\mathcal{O}}$ by the saturation assumption on $\mathcal{R}$. The quotient $\widetilde{\mathcal{O}}$ can be made into an ordered semiring, where the ordering is induced by the ordering on $\mathcal{O}$, and
$\widetilde{x}+\widetilde{y}=\max\{ x,y \}$ if $x \in \ma^>$ or $y \in \ma^{^>}$, and $\widetilde{x}+\widetilde{y}=\widetilde{x+y}$ otherwise.

\smallskip\noindent
For the remainder of this section, $\lambda X$ is the $n$-dimensional Lebesgue measure of $X\subseteq \mathbb{R}^n$.

\begin{definition}\label{defmeasure}
\begin{itemize}
\item[(a)] Let $X\subseteq \mathcal{O}^n$ be a cell.
If $X$ is $n$-fat, then $\nu X = \lambda \st (X)$.  If $X$ is $n$-thin, then $\nu X = \widetilde{a}$, where 
$$a = \max \{ \Pi_{i=1}^{n} (y_i - x_i ) \colon  [x_1 , y_1] \times \dots \times [x_n , y_n ] \subseteq \closure{\tau_X X }  \}.$$


\item[b)] Suppose $X\subseteq \mathcal{O}^n$ is a definable set.  Then $\nu X = \sum_{i=1}^{d}\nu C_i$, where $\{ C_i \}$ is a finite partition of $X$ into cells.
\end{itemize}
\end{definition}
It is shown in \cite{MS} that the above definition is independent of the decomposition of $X$ into cells, and that $\nu X>0$ iff the interior of $X$ in $\mathcal{R}^n$ is nonempty.

\smallskip\noindent
We now define a $d$-dimensional measure $\nu_d$ on the $d$-thin definable subsets of $\mathcal{O}^n$ such that $\nu_d X>0$ whenever $\dim X\geq d$.
On $d$-fat sets one can define a $d$-dimensional measure as Fornasiero, Vasquez do in \cite{fv}. While the proofs in \cite{fv} work only in a sufficiently saturated o-minimal expansion of a real closed field, using Theorem 3.3 on p.244 in \cite{em}, the results from \cite{fv} can be transferred to the general case.

We use the following definitions and theorem of Paw\l ucki \cite{pawlucki} (with slightly modified terminology).  We state these for $R$, but they hold equally well for $\mathcal{R}$ (and we shall use the same terminology for subsets of $\mathcal{R}^n$ as for subsets of $R^n$).

\begin{definition}
Let $f\colon X\to R$, where $X\subseteq R^n$ is open, be $R$-definable and $C^1$, and let $L\in R^{>}$.  Then we say that $f$ is {\bf $L$-controlled\/} if $|\frac{\partial f}{\partial x_i }(a)|\leq L$ for all $i\in \{ 1,\dots ,n  \}$ and almost all $a\in X$.
\end{definition}

\begin{definition}
An open cell $C=(f,g) \subseteq R^n$ will be called a {\bf standard $L$-cell\/} if it is an open interval in the case $n=1$ and, in the case $n>1$, if $p^{n}_{n-1}C$ is a standard $L$-cell and whenever $f$ or $g$ is finite, then it is $L$-controlled. 
\end{definition}

\begin{theorem}\label{pawlucki} Let $X \subseteq R^n$ be open and definable.  Then there are $S_1 , \dots ,S_k \subseteq R^n$ such that $$X=_0 S_1 \dot\cup \dots \dot\cup S_k,$$ where each $S_i$ is a standard $L$-cell after a permutation of coordinates and $L\in R^{>}$ depends only on $n$.
\end{theorem}
It is an exercise left to the reader to derive the following version of the above Theorem for definable subsets of dimension $<n$:

\begin{corollary}\label{pawluckicor}
Let $X\subseteq \mathcal{R}^n$ be definable of dimension $d<n$.  Then there are $S_1 , \dots ,S_k \subseteq \mathcal{R}^n$ such that $$X=_0 S_1 \dot\cup \dots \dot\cup S_k ,$$ and there are permutations of coordinates $\tau_1 , \dots ,\tau_k \colon \mathcal{R}^n \to \mathcal{R}^n$ such that each $\tau_i S_i$ is an $(i_1 , \dots ,i_n )$-cell with $i_j =1$ when $j \leq d$ and $i_j=0$ when $j>d$, and each $p^{n}_{d}\tau_i S_i$ is a standard $L$-cell and each $p^{n}_{m}\tau_i S_i$, where $m>d$, is the graph of an $L$-controlled function.  Furthermore, $L\in \mathcal{O}^{>}$ depends only on $n$.
\end{corollary}
From now on we shall always assume that $L \in \mathcal{O}^{>}$.
We will refer to $(i_1 , \dots ,i_n )$-cells $C$ such that $i_j =1$ for all $j\leq d$ and $i_j = 0$ for $j>d$ and such that $p^{n}_{d}C$ is a standard $L$-cell and each $p^{n}_{m}C$, where $m>d$ is the graph of an $L$-function as \textbf{$L$-cells} of dimension $d$.
For definable $X, S_i \subseteq \mathcal{R}^n$ (or $R^n$) we say that $\{ S_i \}$ is an \textbf{almost partition} of $X$ if $\{ S_i \}$ is finite and $X=_0 \dot\bigcup_i S_i$.

\begin{definition}
Let $X \subseteq \mathcal{O}^n$ be $\mathcal{R}$-definable and $d$-thin.
Let $\{S_1,\ldots, S_k\}$ be an almost partition of $X$ and let $\tau_1 , \dots ,\tau_k \colon \mathcal{R}^n \to \mathcal{R}^n$ be permutations of coordinates such that each $\tau_i S_i$ is an $L$-cell. Then 
$$\nu_d X := \max \{ \nu p^{n}_{d} \tau_i S_i : 1 \leqslant i \leqslant k \}.$$ 
\end{definition}
To show that the definition of $\nu_d$ makes sense, we must show that it does not depend on the choice of $S_i$ and $\tau_i$.
Note that if $\dim(X) < d$ then $\nu_d(X) = 0$.

Below, $\det J \phi (x)$ is the Jacobian determinant of $\phi$ at $x$.

\begin{definition}
Let $X,Y$ be $n$-dimensional definable subsets of $\mathcal{R}^n$.
We say that a map $\phi$ is a \textbf{weak isomorphism} $X \to Y$ if $\phi \colon U \to V$, where $U,V \subseteq \mathcal{R}^n$ are open, is a definable $C^1$-diffeomorphism, $X\subseteq_0 U$, $Y \subseteq_0 V$, $\phi (X\cap U) =_0 Y$, and $\frac{1}{L}\leq |\det J \phi (x)| \leq  L$ for  almost all $x\in U$, where $L \in \mathcal{O}^{>}$.
We say that $X$ and $Y$ are \textbf{weakly isomorphic} if there is a definable weak isomorphism $X \to Y$.
\end{definition}
We now show that $\nu$ is invariant under weak isomorphisms on thin sets.
The proof is a slight variant of the proof of Theorem 5.4 in \cite{MS}.
For the sake of completeness we outline the argument here; see \cite{MS} for more details.
\begin{lemma}\label{inv}
Suppose $C,D \subseteq \mathcal{R}^n$ are definable, $n$-dimensional and $n$-thin.
If $C$ and $D$ are weakly isomorphic, then $\nu(C) = \nu(D)$.
\end{lemma}
\begin{proof}
Let $\phi : C \to D$ be a weak isomorphism.
We assume towards a contradiction that $\nu C > \nu D$.  Let $a\in \ma^{>}$ be such that $\nu D < \widetilde{a}<\nu C$.  We first reduce to the case when $D$ is an open cell and $C=\phi^{-1}(D)$.  The next reduction is to the case when $C$ is a cell and $\tau_D \circ 
\phi (C) \subseteq B$, where $B$ is a box with $\nu B < \widetilde a$.  Next, we find a box $P \subseteq \tau_C C$ with $\nu P > \widetilde{a}$.  We define $$\Phi= \tau_{D} \circ \phi \circ \tau_{C}^{-1}|_P  \colon P \to \Phi (P) \subseteq B.$$  For some $L \in \mathcal{O}^>$, $\frac{1}{L} \leqslant |\det J \Phi (x)| \leqslant L$ for almost all $x\in P$.

 Let $\theta \colon (0,1)^n \to \interior{P}$ be a $C^1$-diffeomorphism with $\det J \theta (x) = b$ for all $x\in (0,1)^n$, where $\widetilde{b}=\nu P$.  Let $\hat{\theta} \colon B \to P'$ be a $C^1$-diffeomorphism onto a box $P'$ such that $\det J \hat{\theta}(x) = \frac{1}{b}$ for all $x \in B$.  Then 
$$\hat{\theta } \circ \Phi \circ \theta \colon (0,1)^n \to \hat{\theta } \circ \Phi \circ \theta (0,1)^n \subseteq P'$$ is such that $$\frac{1}{L}\leqslant |\det \left( J ( \hat \theta \circ \Phi \circ  \theta )  (x) \right)| \leqslant L$$ for almost all $x\in [0,1]^n$.  By Corollary 6.2, p.17 in \cite{thesispaper}, $ \hat \theta \circ \Phi \circ  \theta$ induces a $C^1$-diffeomorphism $\st [0,1]^n \to \st P'$ outside of a subset of $\st [0,1]^n$ of dimension $<n$ with Jacobian determinant bounded by $\frac{1}{\st L}$ and by $\st L$.  But  $\interior{\st P'} = \emptyset$, which is impossible.
\end{proof}
We now show that $\nu_d$ is well-defined.
\begin{lemma}
Let $X\subseteq \mathcal{O}^n$ be definable and $d$-thin.  Let $\{X_i \}$ and $\{ Y_j \}$ be almost partitions of $X$, and let $\tau_i , \tau'_j \colon \mathcal{R}^n \to \mathcal{R}^n$ be permutations of coordinates such that all $\tau_i X_i$ and $\tau'_j Y_j$ are $L$-cells.  Then 
$$\max_i \{ \nu p^{n}_{d} \tau_i X_i   \}=\max_j \{ \nu p^{n}_{d} \tau'_j Y_j   \}.$$
\end{lemma}

\begin{proof}
Let $\mathcal{Z}=\{ Z_1,\ldots, Z_m \}$ be an almost partition of $X$ containing an almost partition of each $X_i $ and $Y_j $.
For each $i$ we have $$\nu p^{n}_{d} \tau_i X_i = \max\{  \nu p^{n}_{d}\tau_i Z_j \colon j\in \{ 1,\dots, m  \} \mbox{ and } Z_j \subseteq X_i  \}.$$  Hence
$$\max_i \{ \nu p^{n}_{d} \tau_i X_i  \}=\max \{ \nu p^{n}_{d} \tau_i Z_{j} : 1 \leqslant j \leqslant m \}$$ and
it suffices to see that for $Z_k \subseteq X_i \cap Y_j$ of dimension $d$, $\nu p^{n}_{d} \tau_i Z_k = \nu p^{n}_{d} \tau'_j Z_k$.  But this follows from the map $\phi \colon p^{n}_{d}\tau_i Z_k \to p^{n}_{d}\tau'_j Z_k$ given by $p^{n}_{d} \tau_i (x)\mapsto p^{n}_{d}\tau'_j (x)$, where $x\in Z_k$, being a weak isomorphism and from Lemma \ref{inv}.


\end{proof}

\smallskip\noindent
\begin{definition}
Let $M$ be an o-minimal field and
$X$ a definable subset of $M^n$.  Let $\mathcal{B}=\{ B_1 ,\dots , B_k \}$ be a collection of pairwise disjoint boxes in $M^n$. We say that $\mathcal{B}$ is an {\em inner approximation of $X$\/} if $B_i \subseteq X$ for each $i$.   We say that $\mathcal{B}$ is an {\em outer approximation of $X$\/} if $X\subseteq_0 \bigcup_{i=1}^{k}B_i$. The volume of $\mathcal{B}$ is $\sum_{i=1}^{k} \mu B_i$.
\end{definition}
The lemma below is Lemma 4.1 in \cite{thesispaper}:
\begin{lemma}\label{Vbox}
Let $X\subseteq \mathcal{O}^n$ be definable and $n$-fat.  Then there is a box $$[a_1 , b_1 ] \times  \dots \times [a_n , b_n ] \subseteq \closure{X}$$ with $a_1 , b_1 ,  \dots , a_n , b_n  \in \mathbb{Q}$.
\end{lemma}
When dealing with $R$-definable families, we shall not make a notational distinction between their realization in $R$ and in $\mathcal{R}$, as which one is meant will always be clear from the context.

\begin{lemma}\label{defapprox}
Let $A=\{ A_t : t\in \mathbb{R}^>  \}$ be a definable family of open subsets of $\mathbb{R}^n$ such that $\lambda (A_t ) \to 0$ as $t\to 0^+$.
\begin{itemize}
\item[a)] Then there is a definable function $h \colon (0,a)_{\mathcal{R}}\to \mathcal{R}^{>}$, where $a\in \mathbb{R}^{>}$, such that each $h (t)$ is the volume of an inner approximation of $A_t$ and $\widetilde{h (t)} = \nu (A_t )$ for all $t\in \ma^{>}$.

\item[b)] 
If $G\colon \mathcal{R} \to \mathcal{R}$ is definable such that $\nu A_{t }< \widetilde{G(t )}$ when $t\in \ma^{>}$, then there is a definable function $H \colon (0,a)_{\mathcal{R}}\to \mathcal{R}^{>}$, where $a\in \mathbb{R}^{>}$, such that each $H (t)$ is the volume of an outer approximation of $A_t$ and $$\nu (A_t )\leqslant \widetilde{H (t)}<\widetilde{G(t)}$$ for all $t\in \ma^{>}$.
\end{itemize}
\end{lemma}
\begin{proof}
Without loss of generality, we may assume that $A_t \subseteq [0,1]^n$ for each $t$.
First note that $\interior{\st A_t }=\emptyset$ for all $t\in \ma^{>}$:
Suppose towards a contradiction that $t\in \ma^{>}$ is such that $\interior{\st A_t } \not=\emptyset$.
By Lemma \ref{Vbox} there is a box $B\subseteq A_t$ such that all vertices have rational coordinates.
For all sufficiently small $s \in \mathbb{R}^>$, $B \subseteq A_s$,
contradiction with $\lambda A_t \to 0$ as $t\to 0^+$ and $\lambda B > 0$.

To prove a), 
we let $\mathcal{D}$ be a decomposition of $\mathbb{R}^{1+n}$ into cells that partitions $A$. Let $D\in \mathcal{D}$ be such that $D\subseteq A$ and $p^{n+1}_{1}D=(0,a)$ for some $a\in \mathbb{R}^{>}$.  Let $D_1 , \dots ,D_k$ be the open cells in $\mathcal{D}$ with $D_i \subseteq A$ and $p^{n+1}_{1}D_i = (0,a)$.  Then $\bigcup_{i=1}^{k}(D_i )_t =_0 A_t$ for all $t\in \ma^{>}$, and there is $i\in \{ 1,\dots ,k \}$ such that $\nu A_t = \nu ((D_i )_t )$ for all $t\in \ma^{>}$:  
Define $h_{i}\colon (0,a)\to [0,1]$ by $$h_{i}(t)=\sup\{ \Pi_{j=2}^{n+1} (b_j - a_j ):\, [a_2 , b_2 ] \times \dots \times [a_{n+1},b_{n+1}]\subseteq  \closure{\big(\tau_{D_{i}} D_{i}\big)_t} \},$$  hence $\nu ((D_{i})_t ) = \widetilde{h_{i}(t)}$ for all $t\in \ma^{>}$.
Since the functions $h_{i}\colon (0,a)\to [0,1]$ are $R$-definable, if $h_{i}(t ) < h_{j}(t )$ for some $t\in \ma^{>}$, then $h_{i}(t ) <  h_{j}(t)$ for all $t\in \ma^{>}$. It follows that for some $i\in \{ 1, \dots ,k   \}$, $\nu A_t = \widetilde{h_{i} (t)}$ for all $t\in \ma^{>}$.  This finishes part a) of the lemma.

\smallskip

To prove b), let $G\colon [0,1]\to [0,1]$ be $R$-definable with $\nu A_{t }< \widetilde{G(t )}$ for $t\in \ma^{>}$.  Without loss of generality, we assume that $A$ is an open cell of the form $(0,f)$.  
The proof is by induction on $n$.  

If $A\subseteq R^{1+1}$, then $A_t = (0,f(t))\subseteq R$ for each $t\in \ma^{>}$, so part b) of the lemma is obvious.
Now suppose the lemma holds for $n \geqslant 1$, and let $A\subseteq R^{1+(n+1)}$.  Let $h$ be as in part a) of the lemma.  Let $\epsilon \in \ma^{>}$.  Then $\nu A_{\epsilon }=\widetilde{h (\epsilon )}$.  Now $\widetilde{h (\epsilon )} < \widetilde{G(\epsilon )}$ implies that for some $p \in \mathbb{Q}^{>}$, $p<1$, $$\widetilde{h (\epsilon )}<\widetilde{h (\epsilon )}^p <\widetilde{G(\epsilon )}.$$  By the proof of the subclaim in Case 1.1 in the proof of Theorem 4.8 in \cite{MS},  there is $l$, depending only on $p$, and there are $R$-definable functions $y_0 , \dots ,y_l \colon (0,a) \to [0,1]$, where $a \in \mathbb{R}^{>}$, such that $$0=y_0 (t)< y_1 (t)< \dots < y_l (t)=1$$ and $$\widetilde{h (t)}\leqslant \sum_{i=1}^{l} \widetilde{y_i (t)}\cdot \nu ( f^{-1}_{t }[y_{i-1}(t),y_i (t)] ) < \widetilde{h (t )}^p $$ for all $t\in \ma^{>}$.  (In the notation of the proof of Theorem 4.8 \cite{MS}, $y_i (t) = h (t)^{(l-i-1)q_3}$ where $i\in \{ 1, \dots ,l-1  \}$, and $q_3 \in \mathbb{Q}^{>}$ depends only on $p$.)

We first find, for each $i$, an $R$-definable function $H_i$ such that $$\widetilde{y_i (t)} \nu f^{-1}_{t}[y_{i-1},y_i ] \leq \widetilde{H_i (t)}<\widetilde{h(t)}^p $$ on $\ma^{>}$, and $H_i (t)$ is the volume of an outer approximation of $f_{t}^{-1}[y_{i-1},y_i ] \times [0,y_i ]$ on $(0,a)$, where $a\in \mathbb{R}^{>}$.
\begin{itemize}
\item Let $i=l$.  Then for all $t\in \ma^{>}$, 
$$\nu f^{-1}_{t}[y_{l-1}(t),y_l (t)]< \widetilde{h (t)}^p .$$
Thus, inductively, there is an $R$-definable function $H_i \colon (0,a) \to [0,1]$, where $a\in \mathbb{R}^{>}$, such that $$\nu f_{t}^{-1}[y_{l-1}(t),y_l (t)] \leqslant \widetilde{ H_i (t) }< \widetilde{h (t)}^p$$ for all $t\in \ma^{>}$, and so that each $H_i (t)$ is the volume of an outer approximation of $f_{t}^{-1}[y_{l-1}(t),y_l (t)]$.
Then $$\widetilde{y_i (t)} \cdot \nu f_{t}^{-1}[y_{i-1}(t),y_i (t)] \leqslant \widetilde{y_i (t)} \cdot \widetilde{ H_i (t) } = \widetilde{H_i (t)}< \widetilde{h (t)}^p $$ for all $t\in \ma^{>}$.  Furthermore, $y_i (t) \cdot H_i (t)$ is the volume of an outer approximation of $f^{-1}_{t}[y_{l-1}(t),y_l (t)]\times [0,y_{i}(t)]$ for all $t\in (0,a)$.

\item Let $i$ be such that $y_i (t) \in \ma^{>}$ and $\nu f^{-1}_{t}[y_{i-1}(t),y_i (t)] < \widetilde{\ma^{>}}$ for some (hence all) $t\in \ma^{>}$.  Then the function assigning 1 to each $t\in (0,1)$ is the volume of an outer approximation of $f^{-1}_{t}[y_{i-1}(t),y_i (t)]$ for all $t\in (0,1)$, and $H_i (t) = y_i (t)$ is as required.

\item Suppose $y_i (t)\in \ma^{>}$ and $\nu f_{i}^{-1}[y_{i-1}(t),y_i (t)] \in \widetilde{\ma^{>}}$ for $t\in \ma^{>}$.  Let $q\in \mathbb{Q}^{>}$ be such that $$\widetilde{y_i (t)} \cdot \nu f_{t}^{-1}[y_{i-1}(t),y_{i}(t)]<\widetilde{h}^{q}<\widetilde{h}^p .$$ 
We shall find an $R$-definable function $d$ with $d(t)\in \ma^{>}$, $\nu f^{-1}_{t}[y_{i-1}(t),y_i (t)]<\widetilde{d(t)}$, and $$\widetilde{y_i (t)}\cdot \nu f_{t}^{-1}[y_{i-1}(t),y_i (t)] \leqslant \widetilde{y_i (t)}\cdot \widetilde{d(t)}\leqslant \widetilde{h (t)}^q $$ for all $t\in \ma^{>}$.



As mentioned above, for each $i\in \{1,\dots ,l-1   \}$, $y_i (t)=h (t)^{r}$ for some $r\in \mathbb{Q}^{>}$.  
Note that either 
$\widetilde{h (t)}^{q-r} \in \ma^{>}$ for all $t\in \ma^{>}$, or  $\widetilde{h (t)}^{q-r}>\widetilde{\ma^{>}}$ for all $t\in \ma^{>}$.  In the first case, we may set $d (t)=h (t)^{q-r}$.  In the second case, we set $d (t)=\sqrt{b(t)}$, where $\widetilde{b(t)}=\nu f^{-1}_{t}[y_{i-1}(t),y_i (t)]$ for all $t\in \ma^{>}$.

Inductively, we now obtain an $R$-definable function $H_i \colon (0,a) \to [0,1]$ such that $$\widetilde{y_i (t)}\cdot \nu f^{-1}_{t}[y_{i-1}(t),y_i (t)] \leqslant \widetilde{y_i (t)}\cdot \widetilde{H_i (t)} < \widetilde{y_i (t)} \cdot \widetilde{d(t)}\leqslant \widetilde{h (t)}^p $$ for all $t\in \ma^{>}$, and so that $y_i (t)\cdot H_i (t)$ is an upper approximation of $$f^{-1}_{t}[y_{i-1}(t),y_i (t)]\times [y_{i-1}(t),y_i (t)]$$ for all $t\in (0,a)$, where $a\in \mathbb{R}^{>}$.

Now $H(t)=\sum_{i=1}^{l} y_i (t) \cdot H_i (t) $ is an upper approximation of $A_t$ for all $t\in (0,a)$, where $a \in \mathbb{R}^{>}$, and 
$$\nu A_t \leqslant \widetilde{H(t)}< \widetilde{G(t)}$$ for all $t\in \ma^{>}$.

\end{itemize}

\end{proof}

\begin{proposition}\label{limitequals1}
Let $\{A_t \colon t\in \mathbb{R}^{>}\}$ be a definable family of open subsets of $\mathbb{R}^n$
such that $\lambda A_t \to 0$ as $t\to 0^+$.
Let $h$ be as in part a) of Lemma \ref{defapprox}.  Then
$$\lim_{t\to 0^+}\frac{\log{\lambda A_t}}{\log h (t)}=1.$$

\end{proposition}

\begin{proof}
Let $p,q \in \mathbb{Q}^{>}$ with $p<1<q$.  Then, for $t\in \ma^{>}$,
$$\widetilde{h (t )}^q < \widetilde{h (t )} < \widetilde{h (t )}^p .$$  By part b) of Lemma \ref{defapprox}, there is a definable $H\colon (0,a) \to [0,1]$, where $a\in \mathbb{R}^{>}$, so that, on $\ma^{>}$,
$$\widetilde{h (t)}^q <\widetilde{h (t )} \leqslant \widetilde{H (t )} < \widetilde{h (t )}^p ,$$ and $H(t)$ is the volume of an upper approximation of $A_t$ for all $t \in (0,a)$.
So, by Lemma \ref{defapprox}, $$h (t) \leqslant \lambda A_t \leqslant H (t)$$ for all $t\in (0,a)$, and hence we get 
$$h (t)^{q}<\lambda A_t < h (t)^{p}$$ for all sufficiently small $t\in \mathbb{R}^>$.  It follows that $$q < \lim_{t \to 0^+} \frac{\log \lambda A_t }{\log{h(t) }} <p.$$
\end{proof}

\begin{corollary}\label{def}
Let $\{A_{p,t} : p \in \mathbb{R}^l, t \in \mathbb{R}^>\}$ be a definable family of open subsets of $\mathbb{R}^n$ such that $\lim_{t \to 0^+} \lambda A_{p,t} = 0$ for all $p \in \mathbb{R}^l$.
Then $$F(p):=\lim_{t\to 0^+}\frac{\log \lambda A_{p,t}}{\log t} \in \mathbb{R}_{\infty}$$ is a definable function of $p$ taking values in $\Lambda_{\infty }$.
If $R$ is polynomially bounded then $F$ takes only finitely many values.
\end{corollary}
\begin{proof}
For each $p \in \mathbb{R}^l$, let $h_p$ be the function whose existence is guaranteed in part a) of Lemma~\ref{defapprox}, considered as a function $\mathbb{R}^> \to \mathbb{R}^>$.
Note that $h_p$ is uniformly definable in $p$.
Proposition~\ref{limitequals1} shows that
$$ \lim_{t \to 0^+} \frac{ \log \lambda A_{p,t} }{ \log h_p(t)} = 1 \quad \text{for all }p \in \mathbb{R}^l. $$
This implies
$$\lim_{t\to 0^+} \frac{\log \lambda A_{p,t}}{\log t} = \lim_{t\to 0^+}\frac{\log h_p (t)}{\log t} \quad \text{for all }p \in \mathbb{R}^l.$$
The Corollary now follows by an application of Proposition~\ref{proposition:millerconseq}.
\end{proof}




\begin{proposition}\label{keyf}
Let $\mathcal{A}=\{A_{p,t} \colon (p,t) \in \mathbb{R}^k \times \mathbb{R}^> \}$ be a definable family of $d$-dimensional subsets of $\mathbb{R}^n$, where $d<n$ and $\mu_d A_{p,t} \to 0$ as $t\to 0+$ for each $p\in \mathbb{R}^k$.  Then $$\lim_{t\to 0} \frac{\log \mu_d A_{p,t}}{\log t}$$ is a definable function of $p$ taking values in $\Lambda_{\infty }$.  If $R$ is polynomially bounded, then this function takes only finitely many values.
\end{proposition}
\begin{proof}
By Proposition~\ref{proposition:millerconseq}, it suffices to find a definable function $f\colon \mathbb{R}^{k} \times \mathbb{R}^>\to \mathbb{R}$ such that 
$$\lim_{t\to 0+}\frac{\log{\nu_{d}A_{p,t}}}{\log t}=\lim_{t\to 0+}\frac{\log{f(p,t)}}{\log t} \quad \text{for all } p \in \mathbb{R}^k.$$
Let $L$ be the constant corresponding to $n$ from Theorem \ref{pawlucki}.
We shall say that a collection of definable sets $\{ C_i  \}$ and a collection of permutations of coordinates $\{ \tau_i \}$ of $\mathbb{R}^n$ are {\em good for $A \subseteq \mathbb{R}^n$\/} if $\{ C_{i}\}$ is an almost partition of $A$ and each $\tau_i C_i$ is an $L$-cell of dimension $d$. 

Corollary \ref{pawluckicor} yields, for each $(p,t) \in \mathbb{R}^{k+1}$, collections $\{ C_i  \}$ and $\{ \tau_i \}$ good for $A_{p,t}$.
Being an almost partition of $A_{p,t}$ as well as being an $L$-cell are first-order properties, hence compactness and the proof of Theorem \ref{pawlucki} yield finitely many collections $$\{ C_{1,1}(p,t), \dots ,C_{1,l(1)} \}, \dots ,\{ C_{m,1}(p,t), \dots ,C_{m,l(m)} \}, $$ of finitely many families of sets $C_{ij}(p,t) \subseteq \mathbb{R}^n$ defined over the same parameters as $\mathcal{A}$ such that for each $(p,t) \in R^{k+1}$ there is $i \in \{ 1, \dots , m \}$ so that $\{ C_{ij}(p,t)\}$ and some set of permutations of coordinates $\{  \tau_j \}$ are good for $A_{p,t}$.

Let $\{ \underline{\tau}_s  \}$ be the set of all the tuples of permutations of coordinates $\mathcal{R}^n \to \mathcal{R}^n$ of length $\max \{ l(1), \dots ,l(m)  \}$ (note that $\{ \underline{\tau}_s \}$ is finite), and let $(\underline{\tau_s } )_j$ be the $j$-th coordinate of $\underline{\tau}_s$.
Let, for each $i\in \{ 1,\dots ,m  \}$ and each $s$, $h_{is}\colon \mathcal{R}^{k+1}\to \mathcal{R}$ be the definable function such that $\widetilde{h_i (p,t)}=\max_{j} \{  \nu p^{n}_{d}(\underline{\tau}_s )_j C_{ij}(p,t) \}$ and whose existence was proved in Lemma \ref{defapprox}.

Then the function $f\colon \mathbb{R}^{k+1}\to \mathbb{R}$ which assigns to $(p,t)$ the value of $h_{is}$ at $(p,t)$, where $is$ is the smallest number in lexicographic order such that $\{ C_{ij}(p,t) \}$ and $\underline{\tau}_s $ are good for $A_{p,t}$, is as required.

\end{proof}

\end{section}

\begin{section}{Proof of Theorem~\ref{theorem:main}}
Let $\mathcal{X} = \{ (X_\alpha,d_\alpha) : \alpha \in \bR^l \}$ be a definable family of metric spaces.
We prove the following:
\begin{theoremno}
The Hausdorff dimension of $(X_\alpha, d_\alpha)$ is a definable function of $\alpha$ which takes values in $\Lambda_\infty$.
If $R$ is polynomially bounded, then the Hausdorff dimension of the elements of $\mathcal{X}$ takes only finitely many values.
\end{theoremno}

\begin{proof}
Applying Corollary 9.3.4 of \cite{erikthesis} there is a partition of $\bR^l$ into definable sets $A,B$, a definable family of sets $\{ Z_\alpha: \alpha \in A\}$ and a definable family of functions $\{ h_\alpha: \alpha \in A\}$ such that:
\begin{enumerate}
\item $h_\alpha$ is a homeomorphism $(X_\alpha, d_\alpha) \rightarrow (Z_\alpha, e)$ for all $\alpha \in A$,
\item if $\beta \in B$, then there is an infinite definable $A \subseteq X_\beta$ such that $(A, d_\beta)$ is discrete.
\end{enumerate}
Any infinite definable set has cardinality $|\bR|$.
Thus if $\beta \in B$, then $(X_\beta, d_\beta)$ contains a discrete subspace of cardinality $|\bR|$ and is therefore not separable.
Thus Fact~\ref{fact:separable} implies $\hdim(X_\beta, d_\beta) = \infty$ for all $\beta \in B$.
We therefore assume that $B$ is empty.
For all $\alpha \in \bR^l$,
we let $d'_\alpha$ be the metric on $Z_\alpha$ given by
$$ d'_\alpha(h_\alpha(x), h_\alpha(y)) = d_\alpha(x,y) \quad \text{for all } x,y \in X_\alpha. $$
Then $(X_\alpha, d_\alpha)$ is isometric to $(Z_\alpha, d'_\alpha)$ for all $\alpha \in \bR^l$.
Note also that $id : (Z_\alpha, d'_\alpha) \rightarrow (Z_\alpha, e)$ is a homeomorphism for all $\alpha \in \bR^l$.
It suffices to prove the theorem for the family $\{ (Z_\alpha, d'_\alpha) : \alpha \in \bR^l\}$.
We therefore prove the theorem under the assumption that the topologies given by $d_\alpha$ and $e$ agree on $X_\alpha$ for all $\alpha \in \bR^l$.
Applying the Trivialization Theorem there are a partition $\{F_1,\ldots,F_n\}$ of $\bR^l$ into definable sets, definable sets $X_1,\ldots, X_n$, and a definable family of functions $\{g_\alpha: \alpha \in \bR^l\}$ such that $g_\alpha : (X_\alpha, e) \rightarrow (X_i, e)$ is a homeomorphism for all $\alpha \in F_i$.
For all $1 \leqslant i \leqslant n$ and $\alpha \in F_i$ we let $d'_\alpha$ be the metric on $X_i$ given by
$$ d'_\alpha(g_\alpha(x), g_\alpha(y)) = d_\alpha(x,y) \quad \text{for all } x,y \in X_\alpha. $$
It suffices to prove the theorem for each definable family $\{ (X_i, d'_\alpha): \alpha \in F_i\}$ separately.
Therefore we assume that $\mathcal{X} = \{ (X, d_\alpha) : \alpha \in \bR^l \}$ for some definable set $X$, and suppose the topology given by $d_\alpha$ agrees with the usual euclidean topology on $X$ for all $\alpha$.
We let $k = \dim(X)$.

We apply induction to $k$.
If $k = 0$, then $X$ is finite and so $\hdim(X, d_\alpha) =0$.
Suppose $k \geqslant 1$.
Let $\lambda$ be the $k$-dimensional Lebesgue measure on $X$.
We let $B_\alpha(p,t)$ be the open $d_\alpha$-ball with center $p \in X$ and radius $t$ and let $B_e(p,t)$ be the open euclidean ball in $\bR^l$ with center $p$ and radius $t$.
Fix $\alpha \in \bR^l$ and $p \in X$.
For all $\delta \in \bR^>$ there is an $\epsilon \in \bR^>$ such that $B_\alpha(p,\epsilon) \subseteq B_e(p,\delta)$.
As $\lambda[ B_e(p,t) \cap X] \to 0$ as $t \to 0^+$ we have $\lambda [B_\alpha(p,t)] \to 0$ as $t \to 0^+$.
Applying Proposition~\ref{keyf} we get a definable function $F: \bR^l \times X \rightarrow \bR_\infty$ such that
$$ F(\alpha,p) = \lim_{t \to 0^+} \frac{ \log \lambda [B_\alpha(p,t)] }{ \log(t) } \quad \text{for all } \alpha \in \bR^l,p \in X. $$
For all $\alpha \in \bR^l$ we let $F_\alpha: X \rightarrow \bR_\infty$ be given by $F_\alpha(p) = F(\alpha, p)$.
Then $F$ takes values in the field of powers of $R$ and if $R$ is polynomially bounded, then $F$ takes only finitely many values.
We let $\{ U_\alpha : \alpha \in \bR^l\}$ be a definable family of open subsets of $X$ such that for all $\alpha \in \bR^l$:
\begin{enumerate}
\item the restriction of $F_{\alpha}$ to $U_\alpha$ is continuous,
\item $\dim(X \setminus U_\alpha) < k$,
\item every open subset of $U_\alpha$ is $k$-dimensional.
\end{enumerate}
By Fact~\ref{fact:union}:
$$ \hdim(X_\alpha, d_\alpha) = \max\{ \hdim(U_\alpha, d_\alpha), \hdim(X_\alpha \setminus U_\alpha, d_\alpha)\} \quad \text{for all } \alpha \in \bR^l. $$
The inductive assumption gives the theorem for $\{(X_\alpha \setminus U_\alpha, d_\alpha) : \alpha \in \bR^l\}$.
It therefore suffices to prove the theorem for $\{ (U_\alpha, d_\alpha): \alpha \in \bR^l\}$.
By 2. above we have $\lambda [B_\alpha(p,t)] > 0$ for all $p \in U_\alpha$ and $t \in \bR^>$.
Thus Proposition~\ref{proposition:logmassdistro} implies
$$ \hdim( U_\alpha, d_\alpha) = \sup \{ F_\alpha(p) : p \in U_\alpha \} \quad \text{for all } \alpha \in \bR^l. $$
Therefore $\hdim(U_\alpha, d_\alpha)$ is a definable function of $\alpha$.
If $R$ is polynomially bounded, then $F$ takes only finitely many values, all in $\Lambda_\infty$, and so $\hdim(U_\alpha,d_\alpha)$ takes only finitely many values as $\alpha$ varies, and each value is an element of $\Lambda_\infty$
\end{proof}
\end{section}

\begin{section}{Bilipschitz Equivalence}
In \cite{valette} Valette considered definable sets equipped with their induced euclidean metrics, and classified them up to bilipschitz equivalence.
He proved the following:
\begin{theoremno}
There are $|\Lambda|$-many definable sets up to definable bilipschitz equivalence.
If $R$ is polynomially bounded, then a definable family of sets has only finitely many elements up to definable bilipschitz equivalence.
\end{theoremno}
One might speculate that Theorem~\ref{theorem:main} is a consequence of a generalization of Valette's theorem to definable metric spaces.
This is not the case:
\begin{fact}
There is a semialgebraic family of metric spaces which contains continuum many elements up to bilipschitz equivalence.
\end{fact}
The collection of Carnot metrics on $\mathbb{R}^k$ naturally forms a semialgebraic family of metric spaces.
Pansu~\cite{pansu} proved that if two Carnot groups are not isomorphic as groups, then the associated Carnot metrics are not bilipschitz equivalent.
It is known that if $k \geqslant 6$, then there are continuum many pairwise non-isomorphic Carnot group operations on $\mathbb{R}^k$.
See 5.5 of \cite{erikthesis} for details and references.
\end{section}

\end{document}